\documentclass[reqno]{amsart}
\usepackage[utf8]{inputenc}
\usepackage[top=1.4in, bottom=1.2in, left=1.2in, right=1.2in]{geometry}

\usepackage{amssymb,amsfonts,amsmath,amsthm,tikz}
\usetikzlibrary{arrows.meta}
\usepackage{enumerate}
\usepackage[colorlinks=true]{hyperref}
\usepackage[alphabetic]{amsrefs}
\usepackage{cmap}
\usepackage{stmaryrd}
\usepackage[sc,osf]{mathpazo}

\newtheorem{thm}{Theorem}%[section]
\newtheorem{cor}[thm]{Corollary}
\newtheorem{prop}[thm]{Proposition}

\theoremstyle{definition}
\newtheorem{defn}[thm]{Definition}

\newtheorem{exmp}[thm]{Example}

\theoremstyle{remark}
\newtheorem{rem}[thm]{Remark}

\newcommand{\Z}{\mathbb Z}

\newcommand{\gl}{\mathfrak{gl}}

\DeclareMathOperator{\KR}{KR}

\newcommand{\lra}[1]{\langle #1 \rangle}
\newcommand{\lrb}[1]{\llbracket #1 \rrbracket}

\tikzset{anchorbase/.style={baseline={([yshift=-0.5ex]current bounding box.center)}}}

\begin{document}

\title{A diagrammatic grading bound for colored $\mathfrak{gl}_N$ link homology}

\author{Hongjian Yang}
\address{Department of Mathematics, Stanford University, Stanford, CA 94305, USA}
\email{yhj@stanford.edu}

\date{\today}

\begin{abstract}
    We establish a grading bound for colored $\mathfrak{gl}_N$ link homology that can be read off from any link diagram. In the context of contact geometry, this gives a bound on the Thurston--Bennequin number of a Legendrian link from $\mathfrak{gl}_N$ link homology, in the spirit of Ng \cite{ng2005legendrian}. 
\end{abstract}

\maketitle

Fix integers $N\ge 2$ and $0\le m\le N$. In this note, we consider the bigraded $m$-colored $\gl_N$ link homology for links in $S^3$. For background, see \cite{ehrig2018functoriality} and the references therein. We follow the grading conventions in \cite{ehrig2018functoriality}. While defined in a combinatorial manner, colored $\gl_N$ homology is difficult to compute in practice, and a complete nontrivial computation is currently only known for the Hopf link and $T(2,2n+1)$ torus knots \cites{wang2025colored,wang2026link}. Our main result is a bound on the supporting grading of colored $\gl_N$ homology that can be read off from any given diagram.

\begin{thm}\label{thm:main thm}
    Let $D$ be a diagram of an oriented link $K$ with $n$ crossings and $k$ local maxima. Then the $m$-colored $\gl_N$ link homology $\KR_{N,m}^{i,j}(K)$ vanishes when \[|j-i-m(N-m)w(D)|>m(N-m)k+\min\{m,N-m\}^2n,\]where $w(D)$ is the writhe of the diagram $D$.
\end{thm}

Several similar bounds have already been established in the context of link homology theories at different levels of generality \cites{ng2005legendrian,ng2008skein,wu2008braids,wu2013colored}. The one most similar to ours is given by Wu \cite{wu2013colored}*{Proposition 1.1}. The novelty of Theorem~\ref{thm:main thm} is that it does not require any special form of the diagram, and the proof is conceptually different from and simpler than the previous one in \cite{wu2013colored}; in particular, it does not rely on the composition product for MOY graphs \cites{wagner2008khovanov,wu2013colored}. The conclusion is also stronger than \cite{wu2013colored}*{Proposition 1.1} when $m>N/2$.

As an immediate corollary, we obtain a Thurston--Bennequin number bound from the $\gl_N$ homology, in the spirit of Ng's result for Khovanov homology \cite{ng2005legendrian}. Note that our bound carries a crossing number correction which is absent from \cite{ng2005legendrian}; see Remark~\ref{rem:lasagna} below.

\begin{cor}\label{cor:legendrian tb bound}
    Let $\mathcal{K}$ be a Legendrian link in the standard contact $S^3$ with underlying smooth link $K$. Then \[\min\{j-i\colon\KR_{N,m}^{i,j}(K)\not=0\}\ge m(N-m)\operatorname{tb}(\mathcal{K})-\min\{m,N-m\}^2c(\mathcal{K}),\]where $\operatorname{tb}(\mathcal{K})$ is the Thurston--Bennequin number of $\mathcal{K}$, and $c(\mathcal{K})$ is the minimal number of crossings of front projections in the Legendrian link type of $\mathcal{K}$.
\end{cor}

\begin{proof}
    Choose a front projection $D$ of $\mathcal{K}$ that realizes the minimal number of crossings in the Legendrian link type of $\mathcal{K}$. Apply Theorem~\ref{thm:main thm} to the diagram obtained by rotating $D$ by $90^\circ$ and rounding the cusps. The conclusion follows by noticing that $k$ is equal to half of the number of cusps in $D$, and $\operatorname{tb}(\mathcal{K})=w(D)-k$.
\end{proof}

\begin{rem}\label{rem:lasagna}
    The initial motivation for studying the grading bound as in Theorem \ref{thm:main thm} was to prove a vanishing theorem for $\gl_N$ skein lasagna modules \cite{morrison2022invariants} in the style of \cite{ren2024khovanov}*{Theorem 1.4}. Unfortunately, this is prevented by the second term in the estimate involving crossing numbers in Theorem \ref{thm:main thm}, which is due to the presence of thick edges in the $\gl_N$ link homology, in contrast to \cite{ng2005legendrian}. In this sense, \cite{ren2024khovanov}*{Theorem 1.4} seems to essentially rely on the unoriented model of Khovanov homology. The same term also keeps Corollary~\ref{cor:legendrian tb bound} weaker than \cite{ng2005legendrian}*{Theorem 1} when $N=2$ and $m=1$, where the latter does not have the crossing number term. 
\end{rem}

%Another corollary is a vanishing theorem for $\mathfrak{gl}_N$ skein lasagna modules in the spirit of \cite{ren2024khovanov}*{Theorem 1.4}.

%\begin{thm}\label{thm:lasagna vanishes}
 %  Let $K$ be a knot in $S^3$. Assume that the knot trace $X_n(K)$ with framing $n\ge -\operatorname{TB}(-K)$ embeds in some $4$-manifold $X$. Then $S^N_0(X,L)=0$ for any $L\subset \partial X$.
%\end{thm}

%\begin{cor}
 %   The $\gl_N$ skein lasagna module of $S^2\times S^2$ vanishes for all $N$.
%\end{cor}

We say a link diagram is \textit{shy} if it has only finitely many local minima and maxima in the $y$ direction, which are away from the crossings. It is clear that any link diagram can be perturbed to a shy diagram by planar isotopy. This is in contrast to the braid position condition: it could take a considerable number of Reidemeister moves to transform a general diagram into a braid closure. The idea behind the proof of Theorem~\ref{thm:main thm} is very simple: resolving crossings in a shy diagram does not increase the number of local maxima, and every simple closed curve in the plane must have at least one local maximum.

A key difference between our setting and Ng's result is the presence of trivalent vertices and, globally, $\gl_N$ webs. Recall that a \textit{$\gl_N$ web} is a planar oriented trivalent graph $\Gamma$ with vertex set $V(\Gamma)$ and edge set $E(\Gamma)$. Each edge $e$ is equipped with a label $l(e)\in\{0,1,\dots,N\}$, subject to the flow condition at each vertex $v\in V(\Gamma)$. For each vertex $v$, let $l(v)$ and $r(v)$ be the thin edges, i.e., the edges with smaller labels, adjacent to $v$, such that $l(v)$ is on the left with respect to the orientation of $\Gamma$. Let $a(v)$ and $b(v)$ be the labels on $l(v)$ and $r(v)$ respectively.

For our purposes, we require that $\gl_N$ webs also carry some geometric structure. 

\begin{defn}
    A $\gl_N$ web $\Gamma$ is said to be \textit{shy} if the following conditions hold.\begin{enumerate}[(a)]
        \item The tangency directions of the three edges adjacent to each trivalent vertex of $\Gamma$ agree. In other words, the underlying curve of $\Gamma$ is a \textit{train track}. 

        \item The tangency direction at each trivalent vertex is not parallel to the $x$ or $y$ axes.

        \item There are only finitely many local extrema in the $y$ direction.
    \end{enumerate}
\end{defn}

\begin{exmp}\label{exmp:shy}
    Figure~\ref{fig:shy theta} shows three diagrams of the planar theta web. The leftmost one is not shy since it is not a train track, and also because the $2$-labeled edge is horizontal. The middle and right diagrams are shy, but only the middle one comes from the resolution of a knot diagram.
\end{exmp}

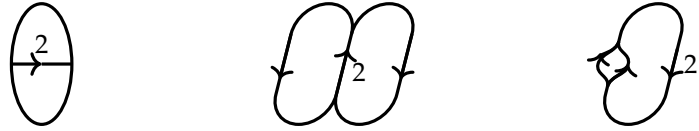
\begin{figure}[hbtp]
        \centering
        \begin{tikzpicture}
            \node at (-4,0) {\begin{tikzpicture}[anchorbase,scale=0.4]
            \draw[very thick] (0,0) ellipse (1cm and 2cm);
            \draw[very thick,->] (-1,0)--(0,0) node[pos=1, above] {$2$};
            \draw[very thick] (0,0)--(1,0);
        \end{tikzpicture}};
        
            \node at (0,0) {\begin{tikzpicture}[anchorbase,scale=0.4]
            \draw[very thick,xslant=0.25] (0,0)--(0,2) to[in=180,out=90] (1,3) to[in=90,out=0] (2,2)--(2,0) to[in=0,out=-90] (1,-1) to[out=180,in=-90] (0,0);
            \draw[very thick,xslant=0.25] (0,2) to[in=0,out=90] (-1,3) to[in=90,out=180] (-2,2)--(-2,0) to[in=180,out=-90] (-1,-1) to[out=0,in=-90] (0,0);
            \draw[very thick,->,xslant=0.25] (0,0)--(0,1.5) node[pos=0.5, right] {$2$};
            \draw[very thick,->,xslant=0.25] (2,2)--(2,0.5);
            \draw[very thick,->,xslant=0.25] (-2,2)--(-2,0.5);
        \end{tikzpicture}};
            \node at (4,0) {\begin{tikzpicture}[anchorbase,scale=0.2]
            
            \draw[very thick,xslant=0.25] (0,0)--(0,0.5) to[in=-90,out=90] (1,2) to[out=90,in=-90] (0,3.5)--(0,4) to[out=90,in=180] (2,6) to[out=0,in=90] (4,4)--(4,0) to[out=-90,in=0] (2,-2) to[out=180,in=-90] (0,0);
            \draw[very thick,xslant=0.25] (0,0.5) to[in=-90,out=90] (-1,2) to[out=90,in=-90] (0,3.5);
            \draw[very thick,->,xslant=0.25] (4,3)--(4,1) node[pos=0.5, right] {$2$};
            \draw[very thick,->,xslant=0.25] (0,0.5) to[in=-90,out=90] (1,2);
            \draw[very thick,>-,xslant=0.25] (-1,2) to[out=90,in=-90] (0,3.5);
        \end{tikzpicture}};
        \end{tikzpicture}
        \caption{Three diagrams of the planar theta web. The unlabeled edges are understood as labeled by $1$. }
        \label{fig:shy theta}
    \end{figure}

By (b), these local extrema must be in the interior of edges. It is clear that every $\gl_N$ web can be perturbed to be shy by planar isotopies. Moreover, if all crossings of a shy diagram are resolved in a shy fashion, every web in the (generalized) cube of resolutions is again shy. 

Given a $\gl_N$ web $\Gamma$, Murakami--Ohtsuki--Yamada \cite{murakami1998homily} defined the evaluation $\lra{\Gamma}_N$ as a state sum \begin{equation}\label{eqn:MOY defn}
\lra{\Gamma}_N=\sum_{\sigma}\left(\prod_{v\in V(\Gamma)}\operatorname{wt}(v,\sigma)\right)q^{\operatorname{rot}(\sigma)}\in\Z[q,q^{-1}],
\end{equation}where\begin{itemize}
        \item $\sigma$ is a \textit{coloring} of $\Gamma$, i.e., it assigns an $l(e)$-element subset of $\{-N+1,-N+3,\dots,N-1\}$ to each edge $e$, subject to the flow condition at each vertex.

        \item $\operatorname{wt}(v,\sigma)$ is the \textit{weight} at $v$, defined as \[q^{a(v)b(v)/2-\pi(\sigma(l(v)),\sigma(r(v)))},\]where $\pi(\sigma(l(v)),\sigma(r(v)))$ is a certain non-negative integer depending on $v$ and $\sigma$.

        \item $\operatorname{rot}(\sigma)$ is the \textit{rotation number}, defined as \[\sum_{C}\sigma(C)\operatorname{rot}(C),\]where the sum is over all simple closed curves $C$ in the \textit{parallel resolution} of $\Gamma$, defined by replacing every edge $e$ in $\Gamma$ with $l(e)$ copies of parallel edges with appropriate colors (cf. \cite{murakami1998homily}*{Figure 1.4}), and $\operatorname{rot}(C)\in\{\pm 1\}$ is the rotation number of $C$.
    \end{itemize}

The following observation provides a diagrammatic upper bound for the quantum grading of the MOY evaluation of a shy $\gl_N$ web. 

\begin{prop}\label{prop:energy bound}
    Let $\Gamma$ be a shy $\gl_N$ web. Assume that $\Gamma$ has $k$ local maxima in the $y$ direction, let $l_1,l_2,\dots,l_k$ be the labels of the edges carrying them (an edge carrying several local maxima being counted once for each of them), and assume that $2l_i\le N$ for every $i$. Define the {\normalfont energy} of $\Gamma$ as \[e(\Gamma)=\sum_{i=1}^kl_i(N-l_i)+\frac{1}{2}\sum_{v\in V(\Gamma)}a(v)b(v).\]Then the quantum grading of the $\gl_N$ evaluation of $\Gamma$ is bounded above by\begin{equation}\label{eqn:energy}
        \max\deg_q\lra{\Gamma}_N\le e(\Gamma).
    \end{equation}
\end{prop}

    When $\Gamma$ arises as an \textit{MOY track}, i.e., a resolution of a braid closure, it refines the bound in \cite{wu2013colored}*{Proposition 4.3}: Wu's bound reads \[\max\deg_q\lra{\Gamma}_N\le kL(N-L)+\frac{1}{2}\sum_{v\in V(\Gamma)}a(v)b(v),\]where $L$ is the average of the $l_i$'s. The Cauchy--Schwarz inequality gives \[kL(N-L)+\frac{1}{2}\sum_{v\in V(\Gamma)}a(v)b(v)\ge\sum_{i=1}^kl_i(N-l_i)+\frac{1}{2}\sum_{v\in V(\Gamma)}a(v)b(v)=e(\Gamma),\]so our bound is sharper with equality when all $l_i$'s are the same.

The condition $2l_i\le N$ is essential for the proof, but it is harmless for our application: in the cube of resolutions of an $m$-colored diagram, every local maximum lies on an edge labeled by $m$, and we will reduce to the case $2m\le N$ by a duality argument at the beginning of the proof of Theorem~\ref{thm:main thm}.

\begin{proof}%[Proof of Proposition~\ref{prop:energy bound}]
    It suffices to prove that for every coloring $\sigma$ of $\Gamma$, the $q$-exponent in the summand associated to $\sigma$ in (\ref{eqn:MOY defn}) is at most $e(\Gamma)$, i.e., \[\sum_{v\in V(\Gamma)}\left(\frac{1}{2}a(v)b(v)-\pi(\sigma(l(v)),\sigma(r(v)))\right)+\sum_{C}\sigma(C)\operatorname{rot}(C)\le e(\Gamma).\]Since $\pi$ is non-negative, the first sum is at most $\frac12\sum_{v\in V(\Gamma)}a(v)b(v)$, so it suffices to prove that \begin{equation}\label{eqn:rot bound}\sum_C\sigma(C)\operatorname{rot}(C)\le\sum_{i=1}^kl_i(N-l_i).\end{equation}

    Denote by $p_1,\dots,p_k$ the local maxima of $\Gamma$, so that $p_i$ lies on an edge $e_i$ of label $l_i$. By the shy conditions (a) and (b), each $p_i$ lies in the interior of $e_i$, and the parallel resolution introduces no new local maxima. For a simple closed curve $C$ in the parallel resolution, let $\mu(C)\in\{p_1,\dots,p_k\}$ be a point where $C$ attains its global maximum in the $y$ direction. Since the $l_i$ push-offs of $e_i$ carry pairwise distinct colors, $\mu^{-1}(p_i)$ consists of $r_i\le l_i$ curves whose colors form a subset $S_i\subseteq\sigma(e_i)$ of cardinality $r_i$. Note that here $r_i$ is not necessarily equal to $l_i$ since a closed curve may have more than one local maximum.

    Every curve in $\mu^{-1}(p_i)$ attains its global maximum at $p_i$, and at a global maximum the disc bounded by the curve lies locally below, so the direction of travel determines the rotation number; since the push-offs of $e_i$ are co-oriented, all curves in $\mu^{-1}(p_i)$ share the same rotation number $\varepsilon_i\in\{\pm1\}$. Then \[\sum_{C\in\mu^{-1}(p_i)}\sigma(C)\operatorname{rot}(C)=\varepsilon_i\sum_{c\in S_i}c\le\max_{\#S=r_i}\sum_{c\in S}c=r_i(N-r_i)\le l_i(N-l_i).\]Here the last inequality follows from the assumption $r_i\le l_i\le N/2$. Summing over $i$ gives \eqref{eqn:rot bound}.
\end{proof}

\begin{exmp}\label{exmp:theta}
    The planar theta web $\theta$ with a $2$-labeled edge has evaluation $\lra{\theta}_N=[N][N-1]$, where $[k]=\frac{q^k-q^{-k}}{q-q^{-1}}$ is the quantum integer. The maximal quantum grading is $\max\deg_q\lra{\theta}_N=2N-3$. Now consider the theta web in the two shy positions depicted in Figure~\ref{fig:shy theta}. In the right one the two $1$-labeled edges run monotonically between the two vertices, so there is a single local maximum and it lies on the $2$-labeled edge; assuming $N\ge4$, Proposition~\ref{prop:energy bound} applies with $k=1$ and $l_1=2$, and gives \[\max\deg_q\lra{\theta}_N\le2(N-2)+1=2N-3,\]which is optimal. In the middle one the $2$-labeled edge is monotone and each of the two $1$-labeled edges carries a local maximum, so $k=2$ and $l_1=l_2=1$, and Proposition~\ref{prop:energy bound} gives only \[\max\deg_q\lra{\theta}_N\le2(N-1)+1=2N-1.\]Thus the estimate genuinely depends on the chosen shy position, and is not always sharp.
\end{exmp}

We are now ready to prove our main theorem. One can see from the proof that it is really an assertion at the \textit{decategorified} level, as no consideration regarding the differentials is involved.

\begin{proof}[Proof of Theorem \ref{thm:main thm}]
    We first reduce to the case $2m\le N$. Write $-K$ for the link $K$ with the orientation of every component reversed, and note that the diagram $-D$ obtained from $D$ by reversing all arrows has the same number of crossings and the same number of local maxima as $D$, and also the same writhe, since the sign of a crossing is unchanged when both of its strands are reversed. Since all components of $K$ carry the same color, the correction term of \cite{wu2011generic}*{Lemma 1.8} vanishes at every crossing, so \cite{wu2011generic}*{Corollary 1.10} gives an isomorphism \[\KR_{N,m}^{i,j}(K)\cong\KR_{N,N-m}^{i,j}(-K)\] preserving the bigrading. As $m(N-m)$ and $\min\{m,N-m\}$ are invariant under $m\mapsto N-m$, it suffices to prove the statement when $m\le N-m$, i.e., $2m\le N$.

    After a planar isotopy supported near the crossings, we may assume that $D$ is shy. The complex $\operatorname{CKR}_{N,m}(D)$ that computes the colored $\gl_N$ homology of $D$ is a tensor product of the Rickard complexes associated to every crossing of $D$. Each summand in $\operatorname{CKR}_{N,m}(D)$ is a graded abelian group $\lrb{\Gamma}_N$ that categorifies the $\gl_N$ evaluation $\lra{\Gamma}_N$ \cite{robert2020closed}, where $\Gamma$ is a shy colored $\gl_N$ web arising from the concatenation of terms in the Rickard complexes for crossings in $D$ in a planar-algebra manner. The difference between the quantum grading and the homological grading shifts by $m(N-m)w(D)$ (cf. \cite{ehrig2018functoriality}*{Definition 3.3}), so it suffices to prove that for each $\Gamma$ arising in the generalized cube of resolutions, we have \[\max\deg_q\lra{\Gamma}_N\le m(N-m)k+m^2n.\]The local maxima of $D$ lie away from the crossings, and resolving a crossing into a term of the Rickard complex creates no new local maxima. Hence the local maxima of $\Gamma$ are exactly the $k$ local maxima of $D$, and each of them lies on an unresolved arc of $D$, which carries the label $m$. As $2m\le N$, Proposition~\ref{prop:energy bound} applies with $l_1=\dots=l_k=m$ and gives \[\max\deg_q\lra{\Gamma}_N\le e(\Gamma)=m(N-m)k+\frac{1}{2}\sum_{v\in V(\Gamma)}a(v)b(v),\]so it suffices to prove that \[\frac{1}{2}\sum_{v\in V(\Gamma)}a(v)b(v)\le m^2n.\]

    \begin{figure}[htbp]
        \centering
        \begin{tikzpicture}[scale=0.25]
            \begin{scope}[xslant=0.25]
                \draw[very thick,->] (0,-2)--(0,0) node[pos=0.5, right,font=\tiny] {$a$};
                \draw[very thick] (0,0) to[out=90,in=-90] (2,2);
                \draw[very thick,->] (2,2)--(2,4) node[pos=0.5,left,font=\tiny] {$a+b-i$};
                \draw[very thick] (2,4) to[out=90,in=-90] (0,6);
                \draw[very thick,>-] (0,6)--(0,8) node[pos=0.5, right,font=\tiny] {$b$};
                \draw[very thick,->] (4,-2)--(4,0) node[pos=0.5, right,font=\tiny] {$b$};
                \draw[very thick] (4,0) to[out=90,in=-90] (2,2);
                \draw[very thick] (4,0) to[out=90,in=-90] (6,2);
                \draw[very thick,->] (6,2)--(6,4) node[pos=0.5,left,font=\tiny] {$i$};
                \draw[very thick] (6,4) to[out=90,in=-90] (4,6);
                \draw[very thick,>-] (4,6)--(4,8) node[pos=0.5, right,font=\tiny] {$a$};
                \draw[very thick] (2,4) to[out=90,in=-90] (4,6);
            \end{scope}
        \end{tikzpicture}
        \caption{A typical term in a Rickard complex.}
        \label{fig:rickard}
    \end{figure}
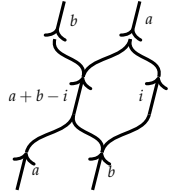

    A typical term (or its mirror) in the Rickard complex for resolving a crossing with two strands labeled $a$ and $b$ has four vertices, as depicted in Figure~\ref{fig:rickard}. In our case $a=b=m$ for each crossing $u$ of $D$, and $0\le i\le m$; note that the label $2m-i$ of the thick edge is at most $N$ because $2m\le N$. The two thin edges at the four vertices are labeled $(m-i,i)$, $(m,m-i)$, $(m,m-i)$ and $(m-i,i)$ respectively. Hence the vertices resolved from $u$ contribute \[\frac{1}{2}\bigl(i(m-i)+m(m-i)+m(m-i)+i(m-i)\bigr)=(m-i)(m+i)=m^2-i^2\le m^2\]to the energy of $\Gamma$. Summing over the $n$ crossings completes the proof.
\end{proof}

\begin{rem}\label{rem:sharper}
    The crossing number term in Theorem~\ref{thm:main thm} is not optimal in general. For instance, when $N=2$ and $m=1$, the standard diagrams of the Hopf link, the right-handed trefoil, and the $(2,5)$-torus knot all show that $n-2$ would already suffice, and is attained. For links with higher braid index, it is possible that this term could be further improved.

    The missing term could come from a refinement of Proposition~\ref{prop:energy bound} of the following shape: if $V(\Gamma)\not=\emptyset$ and the two thin edges at the trivalent vertex $v_0$ of maximal $y$-coordinate lie \emph{above} the thick edge, then \[\max\deg_q\lra{\Gamma}_N\le e(\Gamma)-2a(v_0)b(v_0).\]The heuristic is that the flow condition makes the colors on the two thin edges at $v_0$ disjoint, so that the curves running through them should be counted as a single family of size $a(v_0)+b(v_0)$, replacing $a(N-a)+b(N-b)$ by $(a+b)(N-a-b)$. We do not prove this here: the argument of Proposition~\ref{prop:energy bound} bounds the rotation term and the vertex weights separately, and that is not enough, since the two families of curves in question are assigned to two different local maxima and may well have opposite rotation numbers, in which case the claim above fails for the rotation term alone and can only be recovered from the term $\pi(\sigma(l(v)),\sigma(r(v)))$. For applications to link homology, the assumption of this potential refinement always holds when $m=1$ but not in general.
\end{rem}

\begin{rem}
    One can further derive grading bounds for the $\gl_N$ link homology of links whose components carry different colors, and for colored spatial webs, from Proposition~\ref{prop:energy bound}. In the first case the reduction to $2m\le N$ at the beginning of the proof of Theorem~\ref{thm:main thm} is no longer available in the same form, since the correction term of \cite{wu2011generic}*{Lemma 1.8} need not vanish; one has to keep track of it, or else impose $2l\le N$ on every color $l$ occurring. On the other hand, our strategy of proving Theorem~\ref{thm:main thm} can also recover Ng's results \cites{ng2005legendrian,ng2008skein} without doing induction.
\end{rem}

\subsection*{Note on AI use}Generative AI tools\footnote{The model usage is available upon request; see the statement \href{https://sites.google.com/view/hongjian/research?authuser=0\#h.tvdl6yl4c90r}{here}.} helped find and fix an error in the original proof of Proposition~\ref{prop:energy bound}. They were also used to improve the writing quality of this manuscript.

\subsection*{Acknowledgments}I thank Ciprian Manolescu for his continued encouragement and support. This work was partially supported by the Simons Collaboration grant on New Structures in Low-Dimensional Topology and a Simons Dissertation Fellowship.

%\subsection*{Acknowledgments}This work was partially supported by the Simons Collaboration grant on New Structures in Low-Dimensional Topology and a Simons Dissertation Fellowship.

\bibliographystyle{amsalpha}
\bibliography{ref}

\end{document}